\documentclass[11pt,reqno]{article}
\usepackage{amsmath,amssymb,amsthm}
\usepackage[utf8]{inputenc}
\usepackage[margin=0.9in]{geometry}
\usepackage[dvipsnames]{xcolor}
\usepackage[colorlinks=true,linkcolor=blue,citecolor=magenta]{hyperref}
\usepackage{authblk}
\usepackage{comment}
\usepackage{graphicx}
\usepackage{enumitem}
\usepackage{mlmodern}
\usepackage[T1]{fontenc}

\newtheorem{theorem}{Theorem}[section]
\newtheorem{lemma}[theorem]{Lemma}
\newtheorem{corollary}[theorem]{Corollary}
\newtheorem{definition}[theorem]{Definition}
\newtheorem{proposition}[theorem]{Proposition}

\newcounter{local}
\newcommand{\scl}{\stepcounter{local}}
\newcommand{\bra}[1]{\left(#1\right)}

\newcommand{\z}{\varepsilon}
\newcommand{\ca}{C^{\alpha, \frac{\alpha}{2}}(\bar{Q})}

\newcommand{\eps}{\varepsilon}
\newcommand{\wt}{\tilde}
\newcommand{\wh}{\widehat}

\newcommand{\h}{\hspace{1cm}}

\renewcommand{\L}{\mathsf{L}}
\newcommand{\dH}{d\mathcal H}
\newcommand{\by}{\begin{eqnarray}}
	\newcommand{\ey}{\end{eqnarray}}
\newcommand{\bys}{\begin{eqnarray*}}
	\newcommand{\eys}{\end{eqnarray*}}

\title{Existence, Stability and Optimal Drug Dosage for a Reaction-Diffusion System Arising in a Cancer Treatment}

\author[1]{Jeff Morgan}
\author[2]{Bao Quoc Tang\footnote{Corresponding author.}}
\author[3]{Hong-Ming Yin}
\affil[1]{\small Department of Mathematics, University of Houston, Houston, USA.\break
	\href{mailto:jjmorgan@uh.edu}{jjmorgan@uh.edu}}
\affil[2]{\small Department of Mathematics and Scientific Computing, University of Graz, Graz, Austria\break  \href{mailto:quoc.tang@uni-graz.at}{quoc.tang@uni-graz.at}}
\affil[3]{\small Department of Mathematics and Statistics, Washington State University,
	Pullman, USA\break   \href{mailto:hyin@wsu.edu}{hyin@wsu.edu}}

\date{}

\begin{document}
	\maketitle
	
	\begin{abstract}
		In this paper, a reaction-diffusion system modeling injection of a chemotherapeutic drug on the surface of a living tissue during a treatment for cancer patients is studied. The system describes the interaction of the chemotherapeutic drug and  the normal, tumor and immune cells. We first establish well-posedness for the nonlinear reaction-diffusion system, then investigate the long-time behavior of solutions. Particularly, it is shown that the cancer cells will be eliminated assuming that its reproduction rate is sufficiently small in a short time period in each treatment interval. The analysis is then essentially exploited to study an optimal drug injection rate problem during a chemotherapeutic drug treatment for tumor cells, which is formulated as an optimal boundary control problem with constraints. For this, we show that the existence of an optimal drug injection rate through the boundary, and derive the first-order optimality condition.
		
		
	\end{abstract}
	
	%	\tableofcontents
	
	\section{Introduction}\label{sec1}
	
	Cancer is a major human disease in modern society, for which finding a cure is a great challenge for scientists and medical practitioners. One of the important tasks in fighting this type of disease is to understand how a medical drug interacts with normal, cancer and immune cells. Toward this goal, many researchers have proposed mathematical models based on a system of ordinary differential equations (ODE model) (see, for examples,  \cite{AD1993,BP2000,FGP1999,KP1998,OS1999,WK2014}). One of the popular mathematical models is the following ODE system which describes the interaction of normal, tumor, immune cells and a medical drug, whose concentrations are denoted by $\wt N(t)$, $\wt T(t)$, $\wt I(t)$, and $\wt U(t)$, respectively, (see \cite{BP2000,KP1998,OS1999}):
	\begin{equation}\label{ODEsys}
		\begin{aligned}
			\wt N_t&=r_1\wt N(1-b_1\wt N) -c_4\wt T\wt N-a_3(1-e^{-\wt U})\wt N,\\
			\wt T_t&=r_2\wt T(1-b_2\wt T)-c_2\wt I\wt T-c_3\wt T\wt N-a_2(1-e^{-\wt U})\wt T,\\
			\wt I_t&= s(t)+\frac{\rho \wt I\wt T}{\alpha+\wt T}-c_1\wt I\wt T-k_1\wt I
			-a_1(1-e^{-\wt U})\wt I,\\
			\wt U_t&=v(t)-k_2\wt U.
		\end{aligned}
	\end{equation} 
	Here $s(t)$ and $v(t)$ are the injection rates of
	immune cell and the medical drug, respectively. The values of various parameters in the ODE system \eqref{ODEsys}
	are observed and measured in medical clinics and research
	laboratories (see data in \cite{KMTP1994}). However, it is well observed in experiments and clinical data (\cite{RCM2007}) that cells will diffuse to the surrounding area of living tissue (see Ansarizadeh-Singh-Richards \cite{ASR2017}, Bellomo-Preziosi \cite{BP2000}, Friedman \cite{F2006}, Friedman-Kim \cite{FK2011}, Lou-Ni \cite{NI}, Roose-Chapman-Maini\cite{RCM2007}, Wodarz-Komarova \cite{WK2014}). One can also find many more references in an excellent survey article which gave a summary of various models up to the year 2010 by J. S. Lowengrub, etc. \cite{LF2010}. Therefore, one needs to take a diffusion process into consideration in the mathematical modelling. This fact leads to the following PDE model:
	\begin{equation}\label{PDEsys1}
		\begin{aligned}
			\wh N_t&=\nabla\cdot
			(d_1(x,t)\nabla \wh N)+\wh F_1(\wh N,\wh T,\wh I,\wh U), && (x,t)\in Q_{\tilde{t}}, \\
			\wh T_t&=\nabla\cdot (d_2(x,t)\nabla \wh T)+\wh F_2(\wh N,\wh T,\wh I,\wh U), && (x,t)\in Q_{\tilde{t}},\\
			\wh I_t&=\nabla\cdot (d_3(x,t)\nabla \wh I)+\wh F_3(\wh N,\wh T,\wh I,\wh U), &&(x,t)\in Q_{\tilde{t}},     \\
			\wh U_t&=\nabla\cdot (d_4(x,t)\nabla \wh U)+\wh F_4(\wh N,\wh T,\wh I,\wh U), && (x,t)\in Q_{\tilde{t}},
		\end{aligned}
	\end{equation}
	where $d_i(x,t)$ represents the space-time dependent diffusion coefficient for each type of component, $Q_{\tilde{t}}=\Omega\times (0, \widetilde{t})$ for $\widetilde{t}>0$ and $\Omega$ is a bounded domain in $\mathbb R^n$ with $C^2$-boundary $S=\partial \Omega$.
	
	The interaction functions $\wh F_1,\wh F_2,\wh F_3$ and $\wh F_4$ are similar to those in the ODE
	model \eqref{ODEsys} (see \cite{CSS1992,KMTP1994,KP1998}):
	\bys & & \wh F_1(\wh N,\wh T,\wh I,\wh U)=r_1\wh N(1-b_1\wh N) -c_4\wh T\wh N-a_3(1-e^{-\wh U})\wh N,\\
	& & \wh F_2(\wh N,\wh T,\wh I,\wh U)=r_2\wh T(1-b_2\wh T)-c_2\wh I\wh T-c_3\wh T\wh N-a_2(1-e^{-\wh U})\wh T,\\
	& & \wh F_3(\wh N,\wh T,\wh I,\wh U)= s(x,t)+\frac{\rho \wh I\wh T}{\alpha+\wh T}-c_1\wh I\wh T-k_1\wh I
	-a_1(1-e^{-\wh U})\wh I,\\
	& & \wh F_4(\wh N,\wh T,\wh I,\wh U)=v(x,t)-k_2\wh U.
	\eys
	The  various parameters in $\wh F_i$ are derived for cell-growth
	models from the clinical data (see \cite{KMTP1994}) and $s(x,t)$ is the injection rate of immune cells while $v(x,t)$ represents the injection rate of external chemotherapeutic drug. We remark that $s(x,t)$ can be considered a combination of the natural growth and an injection, and as a result, $s(x,t)$ could be a considered a control variable. However, in this work we will treat $s(x,t)$ as given and focus only on the drug dosage rate $v(x,t)$ as a control.

	\medskip
	For the ODE model \eqref{ODEsys}, De Pillis and Radunskaya in 2003
	(\cite{DR2003}) illustrated a very interesting dynamic of the
	interaction among normal, tumor and immune cells under the treatment of the chemotherapeutic drug. They obtained a range of parameters for
	which the steady-state solution of the system \eqref{ODEsys} is
	stable or unstable. Particularly, they analyzed how the region of
	tumor cells changes under the influence of immune cell and
	chemotherapeutic drug. For the PDE model \eqref{PDEsys1},
	Ansarizadeh-Singh-Richards (\cite{ASR2017}) in 2017
	studied the dynamics of the solution in one-space dimension. They
	obtained similar dynamical stability results for the range of
	various parameters as the ODE model \eqref{ODEsys}. Particularly, they proved the Jeff's phenomenon (see \cite{KMTP1994}) observed in clinical data, which suggests that a tumor might continue to grow after treatment, and then, some time after treatment has ceased, begin to decrease in size, see \cite{DR2001}. Another interesting result was demonstrated in \cite{ASR2017}, where they numerically showed that in order to slow the growth of the tumor, the chemotherapeutic drug should be injected near
	the invasive front of the tumor (optimal location). 
	
	\medskip
	Recently, the author of \cite{Yin2022}
	studied the PDE model \eqref{PDEsys1} for any space dimension. For the nonlinear system \eqref{PDEsys1} subject to appropriate initial-boundary conditions, the global existence and uniqueness are established in \cite{Yin2022}. Particularly, it is shown that under certain conditions the cancer cells will be eliminated after a long-time treatment. Moreover, the author of \cite{Yin2022} also studied the optimal drug dosage problem and proved that there exists an optimal amount of drug during a chemotherapeutic drug treatment for cancer patients.
	
	\medskip
	In this paper we study extended formulation of these existing models, where we assume that the drug will be injected through the surface of a living tissue at periodic intervals. This will definitely be a more effective method during a chemotherapeutic drug treatment. For the corresponding nonlinear reaction-diffusion system, we first prove that there exists a unique global weak solution which is uniformly bounded. Then the asymptotic behavior of the solution to is studied. In particular, it is shown that for the extinction of tumor cells, it is sufficient to require that the reproduction rate of tumor cells is small enough in a short period during each treatment interval (see assumption \eqref{H5}). This somehow aligns well with the observed Jeff's phenomenon \cite{KMTP1994}. Next we turn our attention to finding the optimal drug injection rate during a chemotherapeutic treatment for cancer patients. The resulting  optimal control problem is complicated due to certain constraints that the number of normal and immune cells have to remain above a certain level. We use a penalty method to prove that there exists an optimal drug injection rate for the problem. Moreover, we derive the first order optimality condition for the optimal solution. This condition will provide an effective way to calculate the numerical solution for the optimal control problem.
	
	\medskip
	We end this introduction section by emphasizing that our paper aims at a thorough well-posedness and stability analysis for a RDS model of the form \eqref{PDEsys_main}--\eqref{def_F}, which arises from cancer treatment. These analysis are then exploited to investigate an boundary optimal control problem with constraints. There are plenty of papers modeling tumor-immune actions, where the reaction terms are somewhat different from what is considered in the current paper, which is due specific and different mechanisms, see e.g. \cite{anderson2024global,nikopoulou2021mathematical,lai2017combination} and references therein. We believe that our methods and analysis here can be well adapted to study to those models, and this will be investigated in future research.
	
	\medskip
	
	The paper is organized as follows. In section \ref{sec2}, we describe the generalized PDE model for the above problem. Subsections \ref{sec3} and \ref{sec4} are devoted to the global existence, uniqueness of solutions and their large time stability, respectively. In Section \ref{sec5} we investigate the optimal control problem where we prove the existence of an optimal drug dosage during a chemotherapeutic treatment in subsection \ref{subsec5.1} and derive the optimality condition in subsection \ref{subsec5.2}. Some concluding remarks are given in the final Section \ref{sec6}.

	\section{The Generalized Mathematical Model}\label{sec2}
	In the derivation of the mathematical model, we assume that the concentration for each type of cells satisfies a logistic growth model. 
	Moreover, we assume the interaction between various cells and chemotherapeutic drug have the same pattern as the ODE model \eqref{ODEsys} with the modification that the drug treatment occurs on the boundary, and the immune cells also satisfy a logistic growth.
	With these assumptions, we see that the concentrations of normal, tumor and immune cells satisfy the following reaction-diffusion system\footnote{It is worthwhile to mention that all results in this paper can be extended to the case where the diffusion coefficients depend on all concentrations, i.e. $d_i = d_i(x,t,N,T,I,U)$. We leave the details for the interested reader.}:
	\begin{equation}\label{PDEsys_main}
		\left\{
		\begin{aligned}
			N_t&=\nabla\cdot
			(d_1(x,t,N)\nabla N)+F_1(N,T,I,U), &&(x,t)\in Q_{\tilde{t}},\\
			T_t&=\nabla \cdot(d_2(x,t,T)\nabla T)+F_2(N,T,I,U), && (x,t)\in Q_{\tilde{t}},\\
			I_t&=\nabla \cdot(d_3(x,t,I)\nabla I)+F_3(N,T,I,U), && (x,t)\in Q_{\tilde{t}},\\
			U_t&=\nabla\cdot (d_4(x,t,U)\nabla U)+F_4(N,T,I,U), && (x,t)\in Q_{\tilde{t}},
		\end{aligned}
		\right.
	\end{equation}
	where the nonlinearities $F_1$,$F_2$, $F_3$ and $F_4$ are given as follows: 
	\begin{equation}\label{def_F}
		\begin{aligned}
			&F_1(N,T,I,U) = r_1N(1-b_1N) -c_4TN-a_3(1-e^{-U})N,\\
			&F_2(N,T,I,U) =  r_2T(1-b_2T)-c_2IT-c_3TN-a_2(1-e^{-U})T,\\
			&F_3(N,T,I,U) =  r_3I(1-b_3I)+s(x,t)+\frac{\rho IT}{\alpha+T}-c_1IT-k_1I - a_1(1-e^{-U})I,\\
			&F_4(N,T,I,U) =  -k_2(x,t,N,T,I,U)U,
		\end{aligned}
	\end{equation}
	where $r_i=r_i(x,t,N,T,I,U)$, $i=1,2,3$, is the ``coefficient'' of the logistic growth for each type of cells, $1/b_i$, $i=1,2,3$, represents the reciprocal of the maximum capacity, and the consumption rate $k_2$ for the medical drug is assumed to depend on space-time variables and all types of cells. 
	
	\medskip
	From  clinic practices, the chemotherapeutic drug treatment for patients must be stopped if the concentration of normal cells is below a certain level. Therefore, the drug injection should be stopped if the concentration of normal cells is smaller than a certain level, say, $a_0>0$, where naturally we impose $0<a_0<1/b_1$ which is the capacity of normal cell density. To reflect this fact, we introduce a Heaviside-like function $H(s)$ which is a $C^1$-function with $H(s)=1$ if $s\geq \delta$ for some small $\delta>0$ and $H(s)=0$ if $s< 0$. The drug injection rate $v(x,t)$ through the boundary is given by
	\[  d_4\nabla_{\nu}U(x,t)=v(x,t)H(N-a_0), \;(x,t)\in S\times (0,\infty). \]
	Therefore, for the system \eqref{PDEsys_main} we prescribe the following initial and boundary conditions:
	\begin{equation}\label{BC_IC}
		\left\{
		\begin{aligned}
			&Z(x,0)
			=Z_0(x), \quad Z\in \{N,T,I,U\},\; x\in \Omega,\\
			&\nabla_{\nu} Z(x,t) = 0,  \quad Z\in \{N,T,I\}, \; (x,t)\in S\times (0, \wt{t}),\\
			& d_4\nabla_{\nu}U(x,t))=v(x,t)H(N-a_0),\quad (x,t)\in S\times (0,\wt{t})
		\end{aligned}
		\right.
	\end{equation}
	where $\nu$ is the outward unit normal on $S$ and $\nabla_{\nu}$ represents the normal derivative on $S$.
	
	\medskip
	It is our aim to find an optimal drug injection rate $v(x,t)$ through the boundary $S$ during a chemotherapeutic drug treatment for cancer patients, where the underlying state variables satisfy the nonlinear reaction-diffusion system \eqref{PDEsys_main}--\eqref{BC_IC}, in the sense that a cost function involving the total mass of tumor cells is minimized. This leads us to first study the well-posedness for the nonlinear reaction-diffusion system \eqref{PDEsys_main}--\eqref{BC_IC}. These results are themselves of theoretical interest and also provide the tools in the study of the optimal control problem.
	
	% A weak solution in $V_2(Q):=C([0,\infty);H^1(\Omega))$ for the system \eqref{PDEsys_main}--\eqref{BC_IC} can be defined in a standard way  (see \cite{Evans,LSU}).
	
	\medskip
	In the following subsection, we briefly show that the system \eqref{PDEsys_main}--\eqref{BC_IC} has a unique global weak solution which is uniformly bounded in time. Moreover, the weak solution is also regular if all of the coefficients in the system \eqref{PDEsys_main} are smooth. We also prove in subsection \ref{sec4} that  the solution to the system \eqref{PDEsys_main}--\eqref{BC_IC} converges to the solution of the corresponding steady-state system under certain conditions. In the sequel, we denote $\mathbb R_+=[0,\infty)$ and we use the following function spaces
	\begin{equation*}
		L^{\infty}(Q_{\tilde{t}}):= \{u: Q_{\tilde{t}}\to \mathbb R\,|\; \text{ess sup}_{Q_{\tilde{t}}}|u(x,t)| < +\infty\},
	\end{equation*}
	\begin{equation*}
		\ca:= \{u: \overline{Q_{\tilde{t}}}\to \mathbb R\,|\; u(\cdot,t) \in C^{\alpha}(\bar \Omega) \, \forall t\in [0,\widetilde{t}]\text{ and } u(x,\cdot)\in C^{\frac{\alpha}{2}}([0,\widetilde{t}])\, \forall x\in\bar{\Omega}\}.
	\end{equation*}
	The space $V_2(Q_{\tilde{t}})$ is defined through the norm
	\[ \|u\|_{V_2(Q_{\tilde{t}})}:=\max_{t\in [0,\widetilde{t}]}\|u(t)\|_{L^2(\Omega)}+\|\nabla u\|_{L^2(Q_{\tilde{t}})}.\]
	For brevity, a vector function ${\boldsymbol f}=(f_1,f_2,\cdots,f_m)$ belonging in a product space $B^m$ simply means each component is in the space $B$, and the exponent $m$ is sometimes omitted. 
	
	\subsection{Global Existence, Uniqueness and Regularity}\label{sec3}
	In this section we show that the problem \eqref{PDEsys_main}--\eqref{BC_IC} admits a unique global non-negative bounded weak
	solution under certain minimum conditions on the
	coefficients and known data. 
	
	\medskip
	The following basic conditions are assumed throughout this section. 
	\begin{enumerate}[label=(\textbf{H\theenumi}),ref=\textbf{H\theenumi}]
		\item\label{H1} Let $d_i(x,t,z)$, $i=1,\ldots, 4$, be such that $d_i(x,t,\cdot) \in C^1(\mathbb R_+)$ uniformly in $(x,t)\in Q_{\tilde{t}}$,  $d_i(\cdot,\cdot,z)\in L^\infty(Q_{\tilde{t}})$ for each $z\in \mathbb R$ and $\tilde{t}>0$, and there exists a constant $\delta_0$ such that
		\[ 0<\delta_0\leq d_i(x,t,z), \h (x,t,z)\in Q_{\tilde{t}}\times \mathbb R_+,\;  i=1,\cdots,4.\]
		Also, for each $\tilde{t}>0$, $r_i(x,t,N,T,I,U) \in L^{\infty}(Q_{\tilde{t}}\times \mathbb R_+^4)$ is locally Lipschitz continuous w.r.t. $N,T,I,U$ uniformly in $(x,t)\in Q_{\tilde{t}}$, and
		there exist constants $r_0$ and $R_0$ such that
		\[ 0<r_0\leq r_i(x,t,N,T,I,U)\leq R_0\]
		for all $(x,t,N,T,I,U)\in Q_{\tilde{t}}\times \mathbb R_+^4$.
		
		\item\label{H2} For each $\tilde{t}>0$ the functions $v(x,t)$ and $s(x,t)$ are nonnegative a.e. in $S\times(0,\tilde{t})$ and $Q_{\tilde{t}}$, respectively, and there
		exists a constant $A_1$ such that
		\[ \|v\|_{L^{\infty}(S\times (0,\tilde{t}))}+\|s\|_{L^{\infty}(Q_{\tilde{t}}))}\leq A_1.\]
		In addition, the Heaviside-like function $H(z)\in C^1(\mathbb R)$ is increasing and satisfies $H(z)=0, z\in (-\infty,0]$ and $H(z)=1$ in $[\delta, \infty)$
		for a small constant $\delta>0.$
		
		\item\label{H3} The initial data $(N_0,T_0,I_0,U_0)\in L^{\infty}(\Omega)$ are nonnegative.
		
		\item\label{H4} The parameters $a_i, c_i$ and $k_1,\alpha, \rho$ in the system \eqref{PDEsys_main} are positive.
		Moreover,  $k_2\in L^{\infty}(Q_{\tilde{t}}\times \mathbb R_+^4)$ and  $k_2(x,t,N,T,I,U)\geq k_0>0$ for some constant $k_0$.
	\end{enumerate}

	\medskip
	We start with the definition of a weak solution to \eqref{PDEsys_main}.
	\begin{definition}\label{def_weak_sol}
		A non-negative quadruple $(N,T,I,U)$ is called a non-negative weak solution to \eqref{PDEsys_main}--\eqref{BC_IC} on $Q_{\tilde{t}} = \Omega\times(0,\widetilde{t})$ if for any $Z\in \{N,T,I,U\}$,
		\begin{equation*}
			Z\in V_2(Q)\cap L^{\infty}(Q_{\tilde{t}}), \; \partial_tZ\in L^2(0,\wt{t};(H^1(\Omega))^*),
		\end{equation*}
		\begin{equation*}
			Z(x,0) = Z_0(x), \quad \text{a.e. } x\in\Omega,
		\end{equation*}
		and for any smooth test function $\varphi$ it holds for $Z\in \{N,T,I\}$, $i=1, 2, 3$,
		\begin{equation*}
			\iint_{Q_{\tilde{t}}}\partial_t Z \varphi dxdt + \iint_{Q_{\tilde{t}}}d_i(x,t,Z)\nabla Z\cdot \nabla \varphi dxdt = \iint_{Q_{\tilde{t}}}F_i(N,T,I,U)\varphi dxdt,
		\end{equation*}
		and
		\begin{align*}
			\iint_{Q_{\tilde{t}}}\partial_t U\varphi dxdt + \iint_{Q_{\tilde{t}}}d_4(x,t,U)\nabla U\cdot \nabla \varphi dxdt \\
			= \iint_{S\times(0,\wt{t})}vH(N-a_0)\varphi d\mathcal{H}dt + \iint_{Q_{\tilde{t}}}F_4(N,T,I,U)\varphi dxdt.
		\end{align*}
		A weak solution is called global when it is a weak solution on $Q_{\tilde{t}}$ for any $\widetilde{t}>0$.
	\end{definition}
	The main result of this subsection is the global existence and boundedness of the solution to \eqref{PDEsys_main}--\eqref{BC_IC}. To do that, we use the following lemma which can be of independent interest. The proof of this lemma follows from \cite{AL1979}, and is postponed to the Appendix \ref{appendix1} for completeness.
	\begin{lemma}\label{lem:Linfty}
		Let $t_0>0$ and $0\le u \in L^{\infty}(\Omega\times(0,t_0))\times L^2(0,t_0;H^1(\Omega))$ satisfy
		\begin{equation}\label{h1}
			\begin{cases}
				\partial_t u \le \nabla\cdot (\delta(x,t,u)\nabla u) + au + b, &x\in\Omega, 0<t<t_0,\\
				\delta(x,t,u)\nabla u \cdot \eta \le cu + d(x,t), &x\in\partial\Omega, 0<t<t_0,\\
				u(x,0) = u_0(x), &x\in\Omega,
			\end{cases}
		\end{equation}
		in the weak sense, with $u_0\in L^{\infty}(\Omega)$, $0< \delta_0 \le \delta(x,t,z)$ for all $(x,t,z) \in \Omega\times(0,t_0)\times \mathbb R$, $d\in L^{\infty}(S\times(0,t_0))$, and $a,b,c\ge 0$. Then there is a constant $L>0$ dependent only on $\delta_0, a, b, c, \Omega$, $\|d\|_{L^{\infty}(S\times(0,t_0))}$, $\|u_0\|_{L^{\infty}(\Omega)}$, and $\|u\|_{L^{\infty}(0,t_0;L^1(\Omega))}$, but not explicitly on $t_0$, such that
		\begin{equation*}
			\|u\|_{L^{\infty}(\Omega\times(0,t_0))} \le L.
		\end{equation*}
	\end{lemma}
	\begin{theorem}\label{thm1}
		Assume \eqref{H1}--\eqref{H4}. There exists a global non-negative, bounded weak solution to \eqref{PDEsys_main}, i.e.
		\begin{equation*}
			\|Z\|_{L^{\infty}(\Omega\times (0,\wt{t}))} < \infty, \quad  \forall Z\in \{N,T,I,U\}, \quad \forall \wt{t} > 0.
		\end{equation*}
		In addition, if all functions $d_i,r_i, k_2, s, v$ are smooth and the initial data satisfy the compatibility conditions
		\begin{equation}\label{compatibility_condition}
			\begin{aligned}
				&\nabla_{\nu}Z_0(x) = 0, \; Z\in \{N,T,I\}, \; x\in S, \\
				&d_4\nabla_{\nu}U_0(x) = v(x,0)H(N_0(x)-a_0), \; x\in S,
			\end{aligned}
		\end{equation}
		then the solution is classical in $Q_{\tilde{t}}$ for each $\tilde{t}>0$.
		
		Moreover, if for any $z\in \mathbb R$, and any $\tilde t>0$, the mapping $\Omega\times [0,\tilde t] \ni (x,t)\to d_i(x,t,z)$ is H\"older continuous in each component, then the aforementioned weak solution is unique.
	\end{theorem}
	\begin{proof}
		Let $\eps>0$. We consider the following approximate system with $Z_\eps \in \{N_\eps, T_\eps, I_\eps, U_\eps\}$
		\begin{equation}\label{approx_sys}
			\begin{cases}
				\partial_t Z_{\eps} - \nabla\cdot(d_i^{\eps}(x,t,Z_\eps)\nabla Z_\eps) = F_i^{\eps}(N_\eps,T_\eps,I_\eps,U_\eps), &(x,t)\in Q_{\tilde{t}},\\
				\nabla N_\eps \cdot \nu = \nabla T_\eps \cdot \nu = \nabla I_\eps \cdot \nu = 0, &(x,t)\in S\times(0,\wt{t}),\\
				d_4^{\eps}\nabla U_\eps \cdot \nu = v(x,t)H(N_\eps - a_0), &(x,t)\in S\times(0,\wt{t}),\\
				Z_{\eps}(x,0) = Z_{\eps,0}(x), &x\in\Omega,
			\end{cases}
		\end{equation}
		where $d_i^{\eps}(x,t,\cdot)\in C^{\infty}(Q\times \mathbb R)$ satisfy
		\begin{equation}\label{approx_diff}
			\frac{\delta_0}{2} \le d_i^{\eps}(x,t,z) \le 2A_2, \qquad (x,t,z)\in Q_{\tilde{t}}\times \mathbb R_+, \; i=1,\ldots, 4,
		\end{equation}
		$$\sup_{z\in \mathbb R}\|d_i^{\eps}(\cdot,\cdot,z) - d_i(\cdot,\cdot,z)\|_{L^{\infty}(Q_{\tilde{t}})} \to 0\quad \text{and}\quad \sup_{(x,t)\in Q_{\tilde{t}}}\|d_i^{\eps}(x,t,\cdot) - d_i(x,t,\cdot)\|_{C^1(\mathbb R_+)}\to 0$$ as $\eps \to 0$, with $A_2 = \max_{i=1,\ldots, 4}\sup_{s\in\mathbb R_+}\|d_i(\cdot,\cdot,z)\|_{L^\infty(Q_{\tilde{t}})}$, the approximated nonlinearities given by
		\begin{equation*}
			F_i^{\eps}(N_\eps,T_\eps,I_\eps,U_\eps):= \frac{F_i(N_\eps,T_\eps,I_\eps,U_\eps)}{1 +\eps \sum_{j=1}^4|F_j(N_\eps,T_\eps,I_\eps,U_\eps)|},
		\end{equation*}
		and the initial data $0\le Z_{\eps,0}\in \left(C^{\infty}(\overline\Omega)\right)^4$ satisfies the compatibility \eqref{compatibility_condition} and approximates $Z_0$ in $\left(L^\infty(\Omega)\right)^4$. Since the diffusion coefficients are smooth and the approximated nonlinearities are bounded, the existence of a global weak solution to \eqref{approx_sys} can be obtained using the standard Galerkin method. Moreover, this weak solution is unique, see e.g. \cite{RN}. To show the non-negativity of the approximated solution, we consider the auxiliary system
		\begin{equation*}
			\partial_t Z_{\eps} - \nabla\cdot(d_i^{\eps}(x,t,Z_\eps)\nabla Z_\eps) = F_i^{\eps}(N_{\eps}^+,T_{\eps}^+,I_{\eps}^+,U_{\eps}^+)
		\end{equation*}
		with the same boundary conditions and initial data as in \eqref{approx_sys}, where $z^+ = \max\{z,0\}$. The global existence of a weak solution to this system follows the same way as for \eqref{approx_sys} since the nonlinearities still have all the properties as in \eqref{approx_sys}. By multiplying this equation by $Z_{\eps}^{-} = \min\{Z_{\eps},0\}$ and using the explicit forms of the nonlinearities in \eqref{def_F} which yield $F_i^{\eps}(N_{\eps}^+,T_{\eps}^+,I_{\eps}^+,U_{\eps}^+)Z_{\eps}^{-} \le 0$, we have
		\begin{equation*}
			\frac 12\frac{d}{dt}\|Z_{\eps}^-\|_{L^2(\Omega)}^2 + \delta_0\|\nabla Z_{\eps}^-\|_{L^2(\Omega)}^2 \le 0
		\end{equation*}
		and consequently
		\begin{equation*}
			\|Z_{\eps}^-(t)\|_{L^2(\Omega)}^2 \le \|Z_{\eps,0}^-\|_{L^{2}(\Omega)}^2.
		\end{equation*}
		This shows $Z_{\eps}(t)\ge 0$ for all $t>0$ thanks to the non-negativity of the initial data. Now by uniqueness we obtain that the global weak solution to \eqref{approx_sys} is non-negative.
		
		\medskip
		We now turn to uniform-in-$\eps$ estimates of solutions to \eqref{approx_sys}. Let $0<t_0<\tilde{t}$. From the equation of $U_\eps$ we have
		\begin{equation*}
			\begin{cases}
				\partial_t U_\eps - \nabla\cdot(d_4(x,t,U_\eps)\nabla U_\eps) \le -k_0 U_\eps, &x\in\Omega,\\
				d_4(x,t,U_\eps)\nabla U_\eps \cdot \eta = v(x,t)H(N_\eps - a_0), &x\in S,\\
				U_\eps(x,0) = U_{\eps,0}, &x\in\Omega.
			\end{cases}
		\end{equation*}
		By integrating over $\Omega$, we have
		\begin{equation*} 
			\frac{d}{dt}\int_{\Omega}U_\eps dx + k_0\int_{\Omega}U_\eps dx \le \int_{S}v(x,t)H(N_\eps - a_0)\dH \le C(H, |S|)\|v(t)\|_{L^\infty(S)},
		\end{equation*}
		which leads to
		\begin{equation*}
			\|U_\eps\|_{L^{\infty}(0,T;L^1(\Omega))} \le \|U_{\eps,0}\|_{L^1(\Omega)} + C(H,|S|)\|v\|_{L^\infty(S\times(0,t_0))}.
		\end{equation*}
		Now, we can apply Lemma \ref{lem:Linfty}, we have
		\begin{equation}\label{ULinf}
			\|U_\eps\|_{L^\infty(\Omega\times(0,t_0))} \le C(\|U_0\|_{L^\infty(\Omega)},H, |S|,\|v\|_{L^\infty(S\times(0,t_0))}).
		\end{equation}
		By using similar arguments for the equations of $N_{\eps}, T_{\eps}, I_{\eps}$ with the remark that all superlinear order terms have non-positive sign (see \eqref{def_F}), we get
		\begin{equation}\label{ZLinf}
			\|Z_{\eps}(t)\|_{L^{\infty}(\Omega)} \le C(\wt t) <+\infty \quad \forall t>0, \; Z\in \{N,T,I\}.
		\end{equation}
		By \eqref{ULinf} and \eqref{ZLinf}, $\|F_i(N_\eps,T_\eps,I_\eps,U_\eps)\|_{L^2(Q_{\tilde t})}\le C$, and standard arguments give
		\begin{equation*}
			\|Z_\eps\|_{V(Q)}\le C, \quad Z_\eps \in \{N_\eps,T_\eps,I_\eps,U_\eps\}.
		\end{equation*}
		By Aubin-Lions lemma, we can extract a subsequence, which is still denoted by $(N_\eps,T_\eps,I_\eps,U_\eps)$ such that
		\begin{equation*}
			Z_{\eps} \longrightarrow Z \text{ strongly in } L^2(Q_{\tilde{t}}) \text{ and weakly in } L^2(0,\tilde{t};H^1(\Omega)), \quad Z\in \{N,T,I,U\}.
		\end{equation*}
		For a smooth test function, we can then pass to the limit $\eps\to 0$ in the equations
		\begin{equation*}
			\iint_{Q_{\tilde{t}}}\partial_t Z_\eps \varphi dxdt + \iint_{Q_{\tilde{t}}}d_i^{\eps}(x,t,Z_{\eps})\nabla Z_{\eps}\cdot \nabla \varphi dxdt = \iint_{Q_{\tilde{t}}}F_i^\eps(N_\eps,T_\eps,I_\eps,U_\eps)\varphi dxdt
		\end{equation*}
		for $Z_\eps \in \{N_\eps,T_\eps,I_\eps\}$, $i=1,\ldots, 3$, and
		\begin{align*}
			\iint_{Q_{\tilde{t}}}\partial_t U_\eps\varphi dxdt + \iint_{Q_{\tilde{t}}}d_4^\eps(x,t,U_\eps)\nabla U_\eps\cdot \nabla \varphi dxdt \\
			= \iint_{S\times(0,\wt{t})}vH(N_\eps-a_0)\varphi d\mathcal{H}dt + \iint_{Q_{\tilde{t}}}F_4^\eps(N_\eps,T_\eps,I_\eps,U_\eps)\varphi dxdt,
		\end{align*}
		to finally conclude that $(N,T,I,U)$ is a global non-negative weak solution to \eqref{PDEsys_main}.
		
		\medskip
		If the diffusion coefficients are H\"older continuous in $x$ and $t$, then the uniqueness of the weak solution follows from \cite{RN}.
		
		\medskip
		The standard regularity theory for parabolic equations
		(\cite{Evans,Lieberman,LSU}) implies that if all functions and coefficients in \eqref{PDEsys_main} are smooth, and the initial data are smooth and satisfy the compatibility conditions \eqref{compatibility_condition}, then the solution is also classical.
	\end{proof}

	\subsection{Stability analysis}\label{sec4}
	In this subsection, we first investigate the uniform-in-time boundedness of the weak solution constructed in the previous subsection. Recall that from \eqref{H2} we know $v$ and $s$ are only bounded on each finite time interval. For the solutions to be bounded uniformly in time, we need that the constant $A_1$ in \eqref{H2} to be time independent, i.e.
	\begin{equation*}
		\|v\|_{L^{\infty}(S\times [0,\infty))} + \|s\|_{L^{\infty}(\Omega\times[0,\infty))}\le A_1.
	\end{equation*}
	
	\begin{proposition}\label{prop1}
		Assume \eqref{H1}--\eqref{H4} where $A_1$ in \eqref{H2} is time independent. Then the weak solution to \eqref{PDEsys_main} obtained in Theorem \ref{thm1} is bounded uniformly in time, i.e.
		\begin{equation*}
			\sup_{t\ge 0}\|Z(t)\|_{L^{\infty}(\Omega)} \le M < +\infty, \quad \forall Z\in \{N,I,T,U\}.
		\end{equation*}
	\end{proposition}
	\begin{proof}
		By the comparison principle, thanks to the non-negativity of solutions, we immediately have
		\begin{equation*} 
			\limsup_{t\to\infty}\|N(t)\|_{L^\infty(\Omega)} \le \frac{1}{b_1}, \quad \text{ and } \quad \limsup_{t\to\infty}\|T(t)\|_{L^\infty(\Omega)} \le \frac{1}{b_2}.
		\end{equation*}
		For the bound of $U$, we first show that $\|U\|_{L^\infty(0,t_0;L^1(\Omega))}$ is bounded uniformly in $t_0$. By integrating the equation of $U$ in $\Omega$, we get
		\begin{equation*} 
			\frac{d}{dt}\int_{\Omega}U(x,t)dx + k_0\int_{\Omega}U(x,t)dx = \int_{S}v(x,t)H(N-a_0)dx \le C|S|\|v(t)\|_{L^\infty(S)}.
		\end{equation*}
		The classical Gronwall inequality gives
		\begin{equation*} 
			\int_{\Omega}U(x,t)dx \le e^{-k_0 t}\int_{\Omega}U_0(x)dx + C|S|\sup_{t\ge 0}\|v(t)\|_{L^{\infty}(S)}\frac{1}{k_0}
		\end{equation*}
		which is the desired bound of $U$. Now by applying Lemma \ref{lem:Linfty}, we obtain
		\begin{equation*} 
			\limsup_{t\to\infty}\|U(t)\|_{L^\infty(\Omega)} < +\infty.
		\end{equation*}
		%		
		%		\medskip
		%		Thanks to Young's inequality, assumption \eqref{H2} with $A_2$ being independent of time, and \eqref{e5}, we can estimate
		%		\begin{equation*}
			%			\int_{S}vH(N-a_0)U^{p-1}d\mathcal{H} \le \frac{\delta_0}{2}\int_{\Omega}|U|^{p-2}|\nabla U|^2dx + \frac{k_0}{2}\int_{\Omega}|U|^pdx + C(A_1,k_0,\delta_0,\delta).
			%		\end{equation*}
		%		Therefore
		%		\begin{equation*}
			%			\frac{d}{dt}\|U\|_{L^p(\Omega)}^p + \frac{k_0p}{2}\|U\|_{L^p(\Omega)}^p \le Cp,
			%		\end{equation*}
		%		and consequently Gronwall's lemma implies
		%		\begin{equation*}
			%			\|U(t)\|_{L^p(\Omega)}^p \le e^{-\frac{k_0p}{2}t}\|U_0\|_{L^p(\Omega)}^p + \frac{2C}{k_0}.
			%		\end{equation*}
		%		Taking the root of order $p$ of both sides and letting $p\to \infty$ we get
		%		\begin{equation*}
			%			\sup_{t\ge 0}\|U(t)\|_{L^{\infty}(\Omega)} < +\infty.
			%		\end{equation*}}
	For the boundedness of $I$, we first observe from \eqref{def_F}, \eqref{H1} and \eqref{H2} that 
	\begin{equation*}
		F_3(N,T,I,U)\le R_0I - r_0b_3I^2 + A_1 + \rho I \le \lambda_1I - \lambda_2 I^2 + A_1 \le \widetilde{A} - \widetilde{\lambda}I
	\end{equation*}
	for some $\lambda_1, \lambda_2, \widetilde{A}, \widetilde{\lambda} > 0$. By the comparison principle 
	\begin{equation*} 
		\limsup_{t\to\infty}\|I(t)\|_{L^{\infty}(\Omega)} < +\infty
	\end{equation*}
	which finishes the proof of the Proposition.
\end{proof}

\begin{corollary}\label{corollary}
	Under assumptions \eqref{H1}--\eqref{H4} where $A_1$ in \eqref{H2} is time independent, there exists a maximal attractor $\mathcal{A}$ in $L^{\infty}_+(\Omega)$ for the semigroup $\{S(t)\}_{t\ge 0}$ defined by $S(t): (N_0, T_0, I_0,U_0) \mapsto  (N(t), T(t), I(t),U(t))$ where  $(N,T,I,U)$ is the solution to the system \eqref{PDEsys_main}--\eqref{BC_IC}.
\end{corollary}
\begin{proof}
	It remains to show that the weak solution to \eqref{PDEsys_main} is bounded in $C^{\alpha}(\Omega)$ uniformly in time. For $\tau\ge 0$, we define a smooth cut-off function $\varphi_\tau: \mathbb R \to [0,1]$ such that $\varphi_\tau|_{(-\infty, \tau]} = 0$, $\varphi_\tau|_{[\tau+1,\infty)} = 1$, and $|\varphi_\tau'| \le C$ for $C$ independent of $\tau$. From the equations of \eqref{PDEsys_main}, we get for $Z\in \{N,T,I,U\}$,
	\begin{equation*} 
		\begin{cases}
			\partial_t(\varphi_\tau Z) - \nabla \cdot(d_i\nabla Z) = \varphi_\tau' Z + \varphi_\tau F_i(N,T,I,U), &x\in \Omega, t>\tau,\\
			\nabla_{\nu}(d_i\varphi_\tau Z)= 0, &x\in S, t>\tau, Z\in \{N,T,I\}\\
			\nabla_{\nu}(d_U\varphi_\tau U) = vH(a_0 - N), &x\in S, t>\tau,\\
			(\varphi_\tau Z)(x,\tau) = 0, &x\in\Omega.
		\end{cases}
	\end{equation*}
	
	Since the right hand sides of all equations are bounded uniformly-in-time in the $L^{\infty}$-norm, one can use the De Giorgi-Nash-Moser theory to get (see e.g. \cite{DB1993} Remark 1.1 page 17 and Theorem 1.3 page 78) for any $Z\in \{N,T,I,U\}$,
	\begin{equation*}
		\sup_{t\in [\tau,\tau+1]}\|Z(t)\|_{C^{\alpha}(\bar{\Omega})} \le C\|\varphi_\tau'Z + F_i(N,I,T,U)\|_{L^{\infty}(\Omega\times(\tau,\tau+1))} \quad \le C\quad  \forall \tau \ge 1,
	\end{equation*}
	where the constant $C$ is independent of $\tau$. Therefore,
	\begin{equation*}
		\limsup_{t\to\infty}\|Z(t)\|_{C^{\alpha}(\bar\Omega)} < +\infty, \quad Z\in \{N,I,T,U\}.
	\end{equation*}
	Thanks to the compact embedding $C^{\alpha}(\bar\Omega)\hookrightarrow L^{\infty}_+(\Omega)$, we obtain the existence of a maximal attractor in $L^{\infty}_+(\Omega)$, see e.g. \cite{Temam1988}.
\end{proof}

Finally, we investigate in more detail the large time dynamics of the system when the injection rate decays and for each interval of treatment, the reproduction rate of tumor cells is small for certain amount of time. More precisely, we assume the following.

\begin{enumerate}[label=(\textbf{H\theenumi}),ref=\textbf{H\theenumi}]
	\setcounter{enumi}{4}
	\item\label{H5} It is assumed that
	\begin{equation*}
		\lim_{t\to\infty}\|v(t)\|_{L^{\infty}(S)} = 0, \quad \inf_{(x,t)}s(x,t)\ge \beta > 0,
	\end{equation*}
	and there are positive constants $R_0$, $\mathsf{L}$, $\xi\in (0,\L)$, $K_0\in \mathbb N$ such that for all $j\ge K_0$,
	\begin{equation*}
		\sup_{x\in\Omega,\mathbf Z\in \mathbb R_+^4}|r_2(x,t,\mathbf Z)|\le R_0 \quad \forall t\in (j\L - \xi, j\L).
	\end{equation*}
\end{enumerate}

\medskip
The last condition in \eqref{H5} means that for each interval of treatment $((j-1)\L, j\L)$, the reproduction rate is required to be small only on an interval of size $\xi\in (0,\L)$ which is $(j\L - \xi, j\L)$. This might be explained as the drug takes a certain time, namely the time interval $((j-1)\L, j\L-\xi)$, to start showing effect on the reproduction rate of tumor cells. Consequently, the number of tumor cells can grow, at most linearly (see Theorem \ref{thm:decay_tumor}), on the interval $((j-1)\L, j\L - \xi)$, but eventually decays to zero as $t\to\infty$. This is in fact consistent with Jeff's phenomenon (\cite{DR2001}) mentioned in the introduction, see Figure \ref{Fig1}.

\begin{figure}[ht!]
	%		\begin{center}
		\includegraphics[width=18cm]{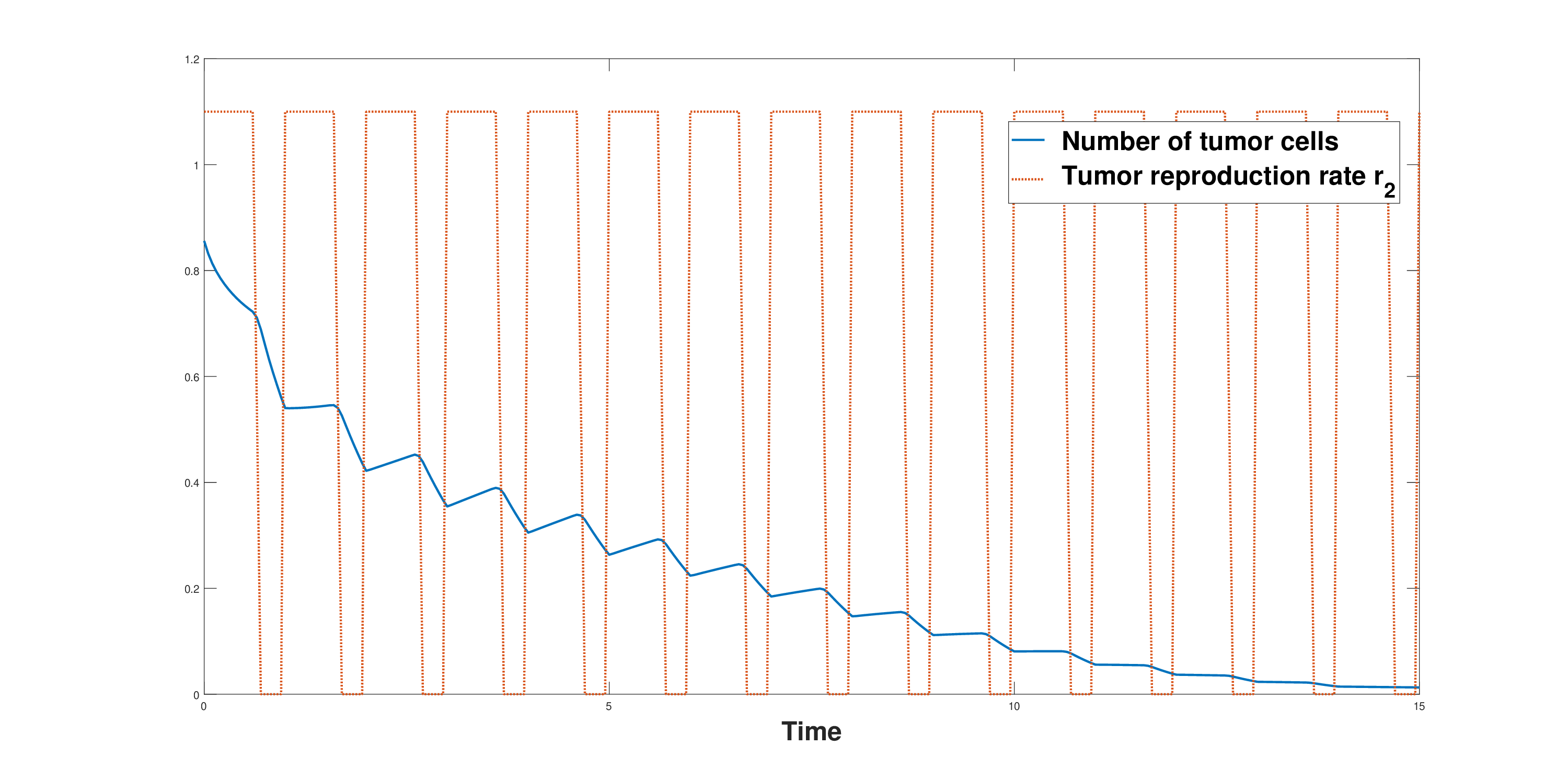}
		%		\end{center}
	\caption{The behavior of number of tumor cells, i.e. the $L^1$-norm of $T$, (the blue dashed line) against its reproduction rate $r_2$ (the red solid line) in the time interval $[0,15]$. Here we choose the treatment interval $\L = 1$. The growth rate $r_2(x,t) \equiv r_2(t)$ depends continuous on time and is defined by: for each natural number $k\in \mathbb N$, $r(t) = 1.1$ for $k\le t\le k+0.6$, $r(t) = 10^{-4}$ for $k+0.7\le t\le k+1$, and $r(t)$ is linear on $(k+0.6, k+0.7)$. Here we notice the Jeff's phenomenon: the number of tumor cells grows in some intervals where the reproduction rate is large, but decays where the reproduction rate is small, and eventually the number of tumors decreases to $0$.}
	\label{Fig1}
\end{figure}
\medskip
\begin{lemma}
	Assume \eqref{H1}--\eqref{H5} with $A_1$ in \eqref{H2} is time independent. Then for any $1\le p <\infty$
	\begin{equation*}
		\lim_{t\to\infty}\|U(t)\|_{L^{p}(\Omega)} = 0.
	\end{equation*}
\end{lemma}
\begin{proof}
	By multiplying the equation of $U$ in \eqref{PDEsys_main} by $U^{p-1}$, $p\ge 2$, and using similar arguments as in Theorem \ref{thm1} we obtain
	\begin{equation*}
		\frac{d}{dt}\|U\|_{L^p(\Omega)}^p \le -Cp(k_0-\|v(t)\|_{L^{\infty}(S)})\|U\|_{L^p(\Omega)}^p + C\|v\|_{L^{\infty}(S)}.
	\end{equation*}
	By using Gronwall's lemma,then applying the decay of $v$ in \eqref{H5} we obtain finally
	\begin{equation*}
		\lim_{t\to\infty}\|U(t)\|_{L^{p}(\Omega)} = 0.
	\end{equation*}
\end{proof}

Finally, we show that the if the reproduction rate of the tumor cells is eventually small then they will decay to zero in large time.
\begin{theorem}\label{thm:decay_tumor}
	Assume \eqref{H1}--\eqref{H5} with $A_1$ in \eqref{H2} time independent. Assume moreover that $\inf_{\Omega}I_0(x)>0$. Then there are positive constants $R_*$ and $\tau_* = \tau_*(R_*)$ such that if $R_0$ in \eqref{H5} satisfies $R_0<R_*$, then 
	\begin{equation*}
		\|T(t)\|_{L^{\infty}(\Omega)}\le Ce^{-\kappa t}, \quad \forall t\ge \tau_*
	\end{equation*}
	for some constants $C, \kappa>0$.
\end{theorem}
\begin{proof}
	From the equation of $I$ in \eqref{PDEsys_main} and \eqref{H5},
	\begin{equation*}
		\partial_t I - \nabla\cdot(d_3(x,t,I)\nabla I) + (r_3b_3I + c_1T + k_1)I = s(x,t) + \frac{\rho IT}{\alpha + T} \ge \beta
	\end{equation*}
	for all $(x,t)\in \Omega\times(0,+\infty)$. Note that $ r_3b_3I + c_1T + k_1$ is bounded uniformly in time and space. Let $d(x,t) = d_3(x,t,I)$, $R(x,t) = r_3b_3I + c_1T + k_1$, $G(x,t) = s(x,t) + \rho IT/(\alpha+T)\ge \beta > 0$, and $K = \|R\|_{L^\infty(Q)}$. Suppose $y(t)$ solves
	\begin{equation*} 
		y'(t) = \beta - Ky(t), \quad y(0) = 0.
	\end{equation*}
	Then the comparison principle implies
	\begin{equation*} 
		I(x,t) \ge y(t) \xrightarrow{t\to\infty} \frac{\beta}{K} > 0 \quad \text{ for all } x\in\Omega.
	\end{equation*}
	
	Thus
	\begin{equation}\label{lower_bound_I}
		\liminf_{t\to\infty} \inf_{x\in \Omega}I(x,t) \ge \gamma > 0.
	\end{equation}
	Turning to the equation of $T$ in \eqref{PDEsys_main} on the time interval $((j-1)\L, j\L-\xi)$, $j\ge K_0$, we have
	\begin{equation*}
		\partial_t T \le \nabla\cdot(d_2(x,t,T)\nabla T) + r_2(x,t,N,T,I,U)T.
	\end{equation*}
	Using the assumption that $r_2 \le R_0$ uniformly in all variables, and thanks to Theorem \ref{thm1}, $$\limsup_{t\to\infty}\|T(t)\|_{L^\infty(\Omega)} \le \omega < +\infty$$
	we can then use comparison principle to get
	\begin{equation*}
		\|T(t)\|_{L^\infty(\Omega)} \le \|T((j-1)\L)\|_{L^\infty(\Omega)} + \omega R_0(t-(j-1)\L), \quad \forall t\in ((j-1)\L, j\L -\xi).
	\end{equation*}
	In particular, it holds
	\begin{equation}\label{ff1}
		\|T(j\L-\xi)\|_{L^\infty(\Omega)} \le \|T((j-1)\L)\|_{L^\infty(\Omega)} + \omega R_0(L-\xi) \quad \forall j\ge K_0+1.
	\end{equation}
	For $t\in (j\L-\xi, j\L)$, we use \eqref{H5} and the lower bound \eqref{lower_bound_I} to get
	\begin{equation*} 
		\partial_t T \le \nabla\cdot(d_2(x,t,T)\nabla T) + (r_2 - c_2I)T \le \nabla\cdot(d_2(x,t,T)\nabla T) + (R_0 - c_2\gamma)T.
	\end{equation*}
	Therefore, for $R_0<R_*:= c_2\gamma$, it follows that
	\begin{equation*} 
		\|T(t)\|_{L^{\infty}(\Omega)}\le e^{-\delta_*(t-(j\L-\xi))}\|T(j\L - \xi)\|_{L^{\infty}(\Omega)}, \quad \forall t\in (j\L-\xi, j\L),
	\end{equation*}
	and in particular
	\begin{equation}\label{f2}
		\|T(j\L)\|_{L^\infty(\Omega)} \le e^{-\delta_*\xi}\|T(j\L-\xi)\|_{L^\infty(\Omega)} \quad \forall j\ge K_0+1.
	\end{equation}
	From \eqref{ff1} and \eqref{f2} we obtain
	\begin{equation*} 
		\|T(j\L)\|_{L^\infty(\Omega)} \le e^{-\delta_*\xi}\|T((j-1)\L)\|_{L^{\infty}(\Omega)} + e^{-\delta_*\xi}\omega R_0(L-\xi) \quad \forall j\ge K_0+1.
	\end{equation*}
	Thanks to this, for large enough $\L$ and $\xi$ sufficiently close to $\L$, we have
	\begin{equation*} 
		\|T(j\L)\|_{L^\infty(\Omega)} \le \frac{1}{2}\|T((j-1)\L)\|_{L^\infty(\Omega)} \quad \forall j\ge K_0,
	\end{equation*}
	which implies the exponential decay of $T(t)$ for large enough $t>0$. 
\end{proof}

\section{The optimal injection rate of chemotherapeutic drug}\label{sec5}

Throughout this section, we assume the nonnegative initial data for \eqref{PDEsys_main} is fixed and bounded. Let $t_0$ be a fixed time and fix a constant $\mathcal A_0$ which is larger than $A_1$ in \eqref{H2} for all $\tilde t$. We introduce an admissible set
\[ U_{ad}=\left\{ v\in L^{\infty}(S\times(0,t_0))\,|\, \mathcal A_0
\ge v(x,t)\geq 0 \; \text{a.e. in }\; S\times (0,t_0)\right\}.\]

For each $v\in U_{ad}$, under assumptions in Theorem \ref{thm1} there exists a unique bounded solution to \eqref{PDEsys_main}, which allows us to define the solution map
\begin{align*}
	\mathcal{S}: U_{ad} &\to \left(L^{\infty}(\Omega\times(0,t_0))\right)^4\\
	v &\mapsto \mathcal{S}(v) = (N,T,I,U).
\end{align*}
In the following, we will write $\mathcal{S}(v)(Z) = Z$ for $Z\in \{N,T,I,U\}$. 
Let
\[ A_0=\int_{\Omega} N_0(x)dx, \quad B_0=\int_{\Omega}I_0(x)dx\]
where $N_0$ and $I_0$ are the given nonnegative initial data, and define
\bys
A(t):=\int_{\Omega}\mathcal{S}(u)(N)(x,t)dx,\\
B(t):=\int_{\Omega}\mathcal{S}(u)(I)(x,t)dx.
\eys

During a chemotherapeutic drug treatment, one of the main goals is to 
find the optimal drug injection rate $v(x,t)$  which will minimize the total amount of tumor cells
at $t_0$. This leads to an optimal control problem in which $v(x,t)$ is the control variable.

For every $v\in U_{ad}$, we define the cost functional as follows:
\[ J(v)=\int_{\Omega}\mathcal{S}(v)(T)(x,t_0) dx+\lambda \|v\|_{L^{\infty}(S\times (0,t_0))},\]
where $\lambda> 0$ is a given regularization parameter.

\medskip
In practice, during a chemotherapeutic drug treatment for cancer patients,  we have to make
sure that the normal and immune cells maintain an
acceptable level. Hence we impose the  constraints
\begin{equation*}
	\int_{\Omega}N(x,t)dx\geq a_0>0\quad\text{and}\quad \int_{\Omega}I(x,t) dx \geq b_0>0\,\text{for}\, t\in[0,t_0],
\end{equation*}
where $0<a_0<A_0$ and $0<b_0<B_0$. It is important to note that for given $t_0>0$, $s(x,t)$ and the set of parameters and initial data in \eqref{PDEsys_main}, it seems not possible to randomly select the values $a_0$ and $b_0$ above. As a result, we assume

\begin{enumerate}[label=(\textbf{H\theenumi}),ref=\textbf{H\theenumi}]
	\setcounter{enumi}{5}
	\item\label{H6} $a_0$ and $b_0$ are values for which there exists at least one $v\in U_{ad}$ so that
	\begin{equation}\label{constraint}
		\int_{\Omega}\mathcal S(v)(N)(x,t)dx\geq a_0>0\quad \text{and}\quad  \int_{\Omega}\mathcal S(v)(I)(x,t) dx \geq b_0>0\;\text{for}\; t\in[0,t_0].
	\end{equation}
\end{enumerate}
Note that thanks to the continuity of solution in time, we can obtain \eqref{H6} if $t_0>0$ is sufficiently small. 

\medskip
Now we can state the following optimal control problem:
%	{\bf Optimal Control Problem}: 
\begin{equation}\label{OCP}\tag{\textbf{P}} 
	\begin{cases}
		&\text{Find $u(x,t)\in U_{ad}$ such that}\\
		&J(u;N,T,I,U) =\inf_{v(x,t)\in
			U_{ad}}J(v;N,T,I,U),\\
		&\text{subject to the constraint \eqref{constraint}.}
	\end{cases}
\end{equation} 

\subsection{Existence of an optimal solution}\label{subsec5.1}

The main result in this section is the following theorem.
\begin{theorem}\label{thm2} Let \eqref{H1}--\eqref{H4} and \eqref{H6} hold. Moreover, for any $z\in \mathbb R$, and any $\tilde t>0$, the mapping $\Omega\times [0,\tilde t] \ni (x,t)\to d_i(x,t,z)$ is H\"older continuous in each component.
	Then there exists a function $u\in U_{ad}$ solving the optimal control problem \eqref{OCP}.
\end{theorem}
\begin{proof} 
	Define 
	\[ M_0=\min\{a_0, b_0\}>0.\] Due to the constraint \eqref{constraint}, we introduce an indicator function (penal function)
	$\beta(z)$:
	\[
	\beta(z)=\left\{ 
	\begin{array}{ll}
		0,    & \mbox{if $z\geq M_0$ }\\
		\infty, & \mbox{otherwise}.
	\end{array}
	\right.
	\]
	The indicator function will ensure that the drug injection will be stopped immediately once
	the amount of either normal cells or immune cells is below the acceptable level. 
	Since the cost functional $J$ is not continuous, we make a smooth approximation with a small $\varepsilon >0$, denoted by $\beta_{\varepsilon}(z)$, for the indicator function $\beta(z)$ by defining a smooth nonincreasing convex function $\beta_\varepsilon(z)$ on $[0,\infty)$ such that $\beta_\varepsilon(z)=0$ if $z\ge M_0+\varepsilon$ and $\beta_\varepsilon(M_0)=\frac{1}{\varepsilon}$. 
	%We use $\beta_{\varepsilon}(s)$ to control the lower $L^{1}(\Omega)$-bound of $I(x,t)$ and $N(x,t)$ for $t\in[0,t_0]$.
	Consider the following approximate cost functional:
	\[ \tilde J(v;N,T,I,U):=J(v;N,T,I,U)+\beta_{\varepsilon}\left(\inf_{t\in[0,t_0]}\left\{\int_\Omega N(x,t)dx,\int_\Omega I(x,t)dx\right\}\right).\]
	
	The cost functional $\tilde J(v; N,T,I,U)$ is nonnegative, but not necessarily convex. However, we can still define
	\[ J_0 =\inf_{v\in U_{ad}}\tilde J(v; N,T,I,U).\]
	From \eqref{H6}, we know each $J_0$ is finite as long as the constrain condition is satisfied. In addition, the construction of $\beta_{\varepsilon}$ implies that for each $n\in \mathbb N$ there exists $u_n\in U_{ad}$ such that 
	\[J_0\le \tilde J(u_n;N,T,I,U)\le  J_0+1/n,\]
	where we choose $\varepsilon=\frac{1}{n}$.
	As a result, 
	\[J(u_n;N,T,I,U)=\tilde J(u_n;N,T,I,U)\]
	for all $n\in \mathbb N$. Let $N_n,T_n,I_n,U_n$ be the solution to \eqref{PDEsys_main} associated with $v=u_n$. The results of the previous section imply this set $\{(N_n,T_n,I_n,U_n)\,|\,n\in \mathbb N\}$ is componentwise nonnegative and uniformly sup norm bounded, and the constrain \eqref{constraint} is satisfied for all $n\in \mathbb N$. Moreover, 
	\[ \lim_{n\to \infty} J(u_n; N_n,T_n,I_n, U_n) = J_0.\]
	Now, due to the construction of $\tilde{J}$ we know the functions $u_n$ are uniformly sup norm bounded.
	Therefore,  there exists a function $u(x,t)\in L^\infty(S_{t_0})$ such that 
	\begin{equation*}
		u_n \rightharpoonup u \quad \text{ weakly in } L^2(S_{t_0})
	\end{equation*}
	and
	\begin{equation*} 
		u_n \overset{*}{\rightharpoonup} u \quad \text{ weakly-* in } L^{\infty}(S_{t_0}).
	\end{equation*}
	So, similar to the proof of Corollary \ref{corollary} in the previous section, we know the sequence $\{(N_n,T_n,I_n,U_n)\}$ is bounded in $C^{\alpha,\alpha/2}(\bar\Omega\times[0,t_0])$, which is compactly embedded in $C(\bar\Omega\times[0,t_0])$.  Therefore, there exists a subsequence, still denoted by $(N_n,T_n,I_n,U_n)$, which converges to
	a limit $(N^*,T^*,I^*,U^*)$  in $C(\bar\Omega\times[0,t_0])$ with respect to the sup norm. 
	Consequently, 
	\[ F_i(N_n,T_n,I_n,U_n)\rightarrow F_i(^*N,T^*,I^*,U^*) \mbox{ in $C(\bar\Omega\times[0,t_0])$ as $n\rightarrow \infty$}, \]
	for each $i=1,2,3,4.$
	Note that a weak solution for $U_n(x,t)$  in $Q_{t_0}$ satisfies the following integral identity:
	\bys
	& & \int_{0}^{t_0}\int_{\Omega}[-U_nw_t+d_4\nabla U_n\cdot \nabla w]dxdt+\int_{0}^{t_0}\int_{S} u_nH(N_n-a_0)w d\mathcal H d\tau\\
	& & =\int_{0}^{t_0}\int_{\Omega} F_4(N_n,T_n,I_n U_n) w dxdt+\int_{\Omega}U_0(x)w(x,0)dx, 
	\eys
	for all test function $w(x,t)\in H^1(0,t_0;L^2(\Omega))\cap L^2(0,t_0;H^1(\Omega))$
	with $w(x,t_0)=0$. Since
	\bys
	& & \nabla U_n\rightharpoonup \nabla U^*, \h \mbox{weakly in $L^2(Q_{t_0})$};\\
	& & F_4(N_n,T_n,I_n,U_n) \rightarrow F_4(N^*,T^*,I^*,U^*), \h \mbox{strongly in $L^1(Q_{t_0})$}.
	\eys 
	It follows that $U^*(x,t)$ is a weak solution to the equation of $U$ in \eqref{PDEsys_main}.  Thus, we see the boundary condition \eqref{BC_IC} for $U^*$ is satisfied. Similarly, the other boundary condition is also satisfied.
	Consequently, we see that $(N^*,T^*, I^*, U^*)$ is a weak solution
	of the system \eqref{PDEsys_main}--\eqref{BC_IC} corresponding to the limit control $u(x,t)$. This concludes the proof of the existence of an optimal control function $u(x,t)$.
\end{proof}

\subsection{The first-order optimality condition}\label{subsec5.2}

In this section we will derive the optimality condition. The crucial step is to find the Gateaux derivatives {of the the cost function $J$ and the solution map $\mathcal S$} with respect to the injection rate $v(x,t)$. { In order to do that, we first show that the solution map $\mathcal S$ is continuous provided the diffusion coefficients are H\"older continuous	
	\begin{lemma}
		Assume \eqref{H1}--\eqref{H4}, and the mapping $\Omega\times [0,\tilde t] \ni (x,t)\to d_i(x,t,z)$ is H\"older continuous in each component for any $z\in \mathbb R$ and any $\tilde t>0$. Then the solution map $\mathcal S$ is continuous from $U_{ad}$ to $(L^{\infty}(\Omega\times(0,t_0)))^4$.
	\end{lemma}
	\begin{proof}
		Let $\{u_n\}_{n\ge 1}$ be a sequence in $U_{ad}$ such that $u_n \to u$ in $L^{\infty}(\Omega\times(0,t_0))$. Denote by $(N_n,T_n,I_n,U_n) = \mathcal{S}(u_n)$. Similarly to the proof of Theorem \ref{thm2}, there is a subsequence $\{n_k\}\subset \{n\}$ such that
		\begin{equation*} 
			(N_{n_k},T_{n_k},I_{n_k},U_{n_k}) \xrightarrow{n_k\to \infty} (N^*,T^*,I^*,U^*) \quad \text{ in } (L^{\infty}(\Omega\times(0,t_0)))^4
		\end{equation*}
		where $(N^*,T^*,I^*,U^*) = \mathcal{S}(u)$. Thanks to the H\"older continuity of the coefficients, we can apply Theorem \ref{thm1} to get the uniqueness of $(N^*,T^*,I^*,U^*)$, and therefore, we have the convergence of the whole sequence $\{(N_n,T_n,I_n,U_n)\}_{n\ge 1}$. This is precisely the continuity of the solution map $\mathcal S$ and the proof is complete.
	\end{proof}
}

Suppose $u(x,t)\in U_{ad}$ is the optimal control and $(N,T,I,U)$ is the corresponding solution
of the system \eqref{PDEsys_main}, i.e. $(N,T,I,U) = \mathcal{S}(u)$. Let
\[w(x,t)=u(x,t)+\z v(x,t)\in U_{ad}, \h \forall \z\geq 0.\]
{Thanks to the continuity of $\mathcal S$}, we can choose $\z$ sufficiently small such that the constrain \eqref{constraint} holds.
\text{Using similar arguments in the proof of Theorem \ref{thm2}, we get}
\[ \lim_{\z\rightarrow 0}(N_{\z}, T_{\z},I_{\z},U_{\z})=(N,T,I,U)\]
in the weak sense of $V_2(Q_{t_0})$ and strongly in $L^2(Q_{t_0})$. Define
\[\hat{N}=\frac{N-N_{\z}}{\z},\quad  \hat{T}=\frac{T-T_{\z}}{\z},\quad \hat{I}=\frac{I-I_{\z}}{\z},\quad \hat{U}=\frac{U-U_{\z}}{\z}.\]
After a routine calculation, we see that $(\hat{N}, \hat{T},\hat{I},\hat{U})$ satisfies the following system in the weak sense:
\begin{equation}\label{F1}
	\begin{aligned}
		\hat{N}_t-\nabla [d_1\nabla \hat{N}] & = & \frac{\partial F_1}{\partial N}(\theta_{N}^{1})\hat{N}+\frac{\partial F_1}{\partial T}(\theta_{T}^{1})\hat{T}+\frac{\partial F_1}{\partial I}(\theta_{I}^{1})\hat{I}+\frac{\partial F_1}{\partial U}(\theta_{U}^{1})\hat{U},\\
		\hat{T}_t-\nabla [d_2\nabla \hat{T}] & = & \frac{\partial F_2}{\partial N}(\theta_{N}^{2})\hat{N}+\frac{\partial F_2}{\partial T}(\theta_{T}^{2})\hat{T}+\frac{\partial F_2}{\partial I}(\theta_{I}^{2})\hat{I}+\frac{\partial F_1}{\partial U}(\theta_{U}^{2})\hat{U},\scl\\
		\hat{I}_t-\nabla [d_3\nabla \hat{I}] & = & \frac{\partial F_3}{\partial N}(\theta_{N}^{3})\hat{N}+\frac{\partial F_3}{\partial T}(\theta_{T}^{3})\hat{T}+\frac{\partial F_3}{\partial I}(\theta_{I}^{3})\hat{I}+\frac{\partial F_3}{\partial U}(\theta_{U}^{3})\hat{U},\scl\\
		\hat{U}_t-\nabla [d_4\nabla \hat{U}] & = & \frac{\partial F_4}{\partial N}(\theta_{N}^{4})\hat{N}+\frac{\partial F_4}{\partial T}(\theta_{T}^{4})\hat{T}+\frac{\partial F_4}{\partial I}(\theta_{I}^{4})\hat{I}+\frac{\partial F_4}{\partial U}(\theta_{U}^{4})\hat{U},\scl
	\end{aligned}
\end{equation}
{where $\theta^j_Z$ is in between $Z$ and $Z_\eps$ for $Z\in \{N,T,I,U\}$}. Moreover, $(\hat{N}, \hat{T},\hat{I},\hat{U})$ satisfies the initial and boundary conditions
\begin{equation}\label{F2}
	\begin{aligned}
		& (\hat{N}, \hat{T},\hat{I},\hat{U})|_{t=0}=(0,0,0,0), \h x\in \Omega,\scl\\
		&\nabla_{\nu}(\hat N,\hat T,\hat I)=(0,0,0), \h (x,t)\in S\times (0,t_0),\scl\\
		&d_4\nabla_{\nu}\hat U=v(x,t)H(N_{\z}-a_0)+uH'(\theta)\hat{N},\scl
	\end{aligned}
\end{equation}
where $H'$ is the derivative of $H$ and, $\theta$  is a mean-value between $N-a_0$ and $N_{\z}-a_0$. Thanks to the boundedness of $Z$ and $Z_{\eps}$, $Z\in \{N,T,I,U\}$, we have for $i=1,2,3,4$,
\[ \frac{\partial F_i}{\partial N}(\theta^i_N),\; \frac{\partial F_i}{\partial T}(\theta^i_T),\;\frac{\partial F_i}{\partial I}(\theta^i_I),\;\frac{\partial F_i}{\partial U}(\theta^i_U)\]
are all bounded. We can use the energy method to obtain
\[ \|\hat{N}\|_{V_2(Q_{t_0})}+\|\hat{T}\|_{V_2(Q_{t_0})}+\|\hat{I}\|_{V_2(Q_{t_0})}+\|\hat{U}\|_{V_2(Q_{t_0})}\leq C,\]
where $C$ depends only on known data, but not on $\z$. The compactness argument implies that there exists $(N^*,T^*,I^*,U^*) $ such that, up to a subsequence,
\[ \lim_{\z\rightarrow 0}(\hat{N},\hat{T},\hat{I},\hat{U})=(N^*,T^*,I^*,U^*) \quad {\text{strongly in } L^2(\Omega\times(0,t_0)).}\]
It follows from $H'(\theta)=\int_{0}^{1}H'(\tau N_{\z}+(1-\tau)N-a_0)d\tau$ that
\[\lim_{\z\rightarrow 0}H'(\theta)=H'(N-a_0) \quad \text{ in } L^2(S\times(0,t_0)).\]

We summarize the above derivation to obtain the following optimality condition.
\begin{theorem} Let $u(x,t)$ be the optimal injection rate and $(N,T,I,U)$ is the corresponding solution of the system \eqref{PDEsys_main}. Then for any $v(x,t)\in U_{ad}$, there exists a vector function
	$(N^*,T^*,I^*,U^*)$ which is a weak solution to the linear system \eqref{F1}--\eqref{F2} with the last condition in \eqref{F2} replaced by
	\[ d_4\nabla_{\nu}U^*=v(x,t)H(N-a_0)+uH'(N-a_0)N^*.\]
\end{theorem}

\section{Conclusion}\label{sec6} 
{\normalfont In this paper we studied a reaction-diffusion system arising from a cancer treatment, where the drug is injected from the boundary of the domain under the condition that number of normal cells is above a certain level. We first show the global existence of a unique bounded weak solution by using energy method, then show that there exists a maximal attractor. Moreover, if the reproduction rate of tumor cells is eventually small, then it is proved that they die out in large time, despite possible temporal increasing due to Jeff's phenomenon. Finally, an optimal control problem is posed with the aim of finding the optimal dosage of the drug is investigated. By a penalty method, we showed the existence of an optimal solution. This result could be potentially used by medical practitioners to design an automatic drug deliver device during a tumor treatment for patients. More study with clinical data is needed for potential applications in clinics, for instance the presence of Jeff's phenomenon or numerical simulations with real clinical data.}

\medskip
\noindent{\bf Acknowledgement.} 
{\normalfont B.Q. Tang received funding from the FWF project ``Quasi-steady-state approximation for PDE'', number I-5213. This work is partially supported by NAWI Graz.}

\appendix
\section{Proof of Lemma \ref{lem:Linfty}}\label{appendix1}
The proof uses the ideas from \cite{AL1979}. For $k\in \mathbb N$, multiplying \eqref{h1} by $u^{2^k-1}$ leads to
\begin{align*} 
	\frac{1}{2^k}\frac{d}{dt}\int_{\Omega}u^{2^k}dx &+ \frac{2^k-1}{2^{2k-2}}\delta_0\int_{\Omega}\left|\nabla(u^{2^{k-1}}) \right|^2dx\\
	&\le a\int_{\Omega}u^{2^k}dx + b\int_{\Omega}u^{2^k-1}dx + c\int_{S}u^{2^k}d\mathcal{H} + \|d\|_{L^{\infty}(S\times(0,T))}\int_{S} u^{2^k-1}d\mathcal{H}.
\end{align*}
Using 
\begin{equation*} 
	b\int_{\Omega}u^{2^k-1}dx \le \frac{b|\Omega|}{2^k} + \frac{b(2^k-1)}{2^k}\int_{\Omega}u^{2^k}dx,
\end{equation*}
\begin{equation*} 
	\|d\|_{L^{\infty}(S\times(0,T))}\int_S u^{2^k-1}\dH \le \frac{|S|\|d\|_{L^{\infty}(S\times(0,t_0))}}{2^k} + \frac{\|d\|_{L^{\infty}(S\times(0,t_0))}(2^k-1)}{2^k}\int_S u^{2^k}\dH,
\end{equation*}
we get
\begin{equation*} 
	\frac{1}{2^k}\frac{d}{dt}\int_{\Omega}u^{2^k}dx + \frac{2^k-1}{2^{2k-2}}\delta_0\int_{\Omega}\left|\nabla(u^{2^{k-1}}) \right|^2dx \le C_0\int_{\Omega}u^{2^k}dx + C_1\int_{S}u^{2^k}\dH + \frac{C_2}{2^{k-1}}
\end{equation*}
where $C_0, C_1, C_2$ depend only on $a, b, c, |\Omega|, |S|$ and $\|d\|_{L^{\infty(S\times(0,t_0))}}$. It follows from this
\begin{equation*}
	\frac 12\frac{d}{dt}\int_{\Omega}u^{2^k}dx + \delta_0 \int_{\Omega}\left|\nabla(u^{2^{k-1}}) \right|^2dx \le C_02^{k-1}\int_{\Omega}u^{2^k}dx + C_12^{k-1}\int_S u^{2^k}\dH + C_2.
\end{equation*}
Denote by $u_k:= u^{2^{k-1}}$ and $C_k:= \max\{C_0,C_1\}2^{k-1}$ we get
\begin{equation}\label{h2}
	\frac 12\frac{d}{dt}\int_{\Omega}u_k^2dx + \delta_0\int_{\Omega}|\nabla u_k|^2dx \le C_k\int_{\Omega}|u_k|^2dx + C_k\int_{S}|u_k|^2\dH + C_2.
\end{equation}
By the Gagliardo-Nirenberg inequality, for any $\eps_0>0$, it holds
\begin{equation*}
	\int_{\Omega}|v|^2dx \le \eps_0 \int_{\Omega}|\nabla v|^2dx + C_{\Omega,n}\eps_0^{-n/2}\bra{\int_{\Omega}|v|dx}^{2},
\end{equation*}
for some $C_{\Omega,n}$ depending only on $\Omega$ and $n$. We combine this with the trace interpolation inequality for any $\eps>0$
\begin{equation*}
	\begin{aligned}
		\int_{S}|v|^2\dH &\le \eps \int_{\Omega}|\nabla v|^2dx + \frac{C_{\Omega,n}}{\eps}\int_{\Omega}|v|^2dx\\
		&\le \bra{\eps + \frac{C_{\Omega,n}\eps_0}{\eps}}\int_{\Omega}|\nabla v|^2dx + \frac{C_{\Omega,n}^2 \eps_0^{-n/2}}{\eps}\bra{\int_{\Omega}|v|dx}^2.
	\end{aligned}
\end{equation*}
Now we choose 
\begin{equation*} 
	\eps = \frac{\delta_0}{8}C_k^{-1} \quad \text{ and } \eps_0 \le \frac{\delta_0 \eps}{8C_{\Omega,n} C_k} = \bra{\frac{\delta_0}{8}}^2\frac{1}{C_{\Omega,n}} C_k^{-2}.
\end{equation*}
Note that by enlarging $C_{\Omega,n}$ sufficiently, we have $\eps_0 C_k \le \frac{\delta_0}{4}$. Therefore
\begin{align*}
	C_k\int_{\Omega}|u_k|^2dx + C_k\int_{S}|u_k|^2\dH &\le \frac{\delta_0}{4}\int_{\Omega}|\nabla u_k|^2dx + C_{\Omega,n}\bra{C_k^{n+1} + C_k^{n+2}}\bra{\int_{\Omega}|u_k|dx}^2 + C_2.
\end{align*} 
Thus we have
\begin{equation*} 
	\frac 12\frac{d}{dt}\int_{\Omega}|u_k|^2dx + \frac{\delta_0}{2}\int_{\Omega}|\nabla u_k|^2dx \le \widehat{C}_k\bra{\int_{\Omega}|u_k|dx}^2 + C_2
\end{equation*}
where $\widehat{C}_k \sim (2^{n+2})^{k}$. Adding both sides with $(\delta_0/2)\int_{\Omega}|u_k|^2dx$ and using the Gagliardo-Nirenberg inequality again, we arrive at
\begin{equation*}
	\frac{d}{dt}\int_{\Omega}|u_k|^2dx + \delta_0\int_{\Omega}|u_k|^2dx \le \widetilde{C}_k\bra{\int_{\Omega}|u_k|dx}^2 + C_2
\end{equation*}
with $\widetilde{C}_k \sim (2^{n+2})^{k}$. An application of the classical Gronwall inequality yields for any $t\in (0,t_0)$
\begin{equation*}
	\begin{aligned} 
		\int_{\Omega}|u_k(t)|^2dx &\le e^{-\delta_0 t}\int_{\Omega}|u_k(0)|^2dx + \widetilde{C}_k\int_0^te^{-\delta_0(t-s)}\bra{\int_{\Omega}|u_k(s)|dx}^2ds + \frac{C_2}{\delta_0}\\
		&\le \int_{\Omega}|u_k(0)|^2dx + \frac{\widetilde{C}_k}{\delta_0}\bra{\sup_{s\in (0,t_0)}\int_{\Omega}|u_k(s)|dx}^2 + \frac{C_2}{\delta_0}.
	\end{aligned}
\end{equation*}
Replacing $u_k = u^{2^{k-1}}$ again, and taking the root of order $2^k$ we obtain
\begin{equation*} 
	\bra{\sup_{t\in(0,t_0)}\int_{\Omega}u^{2^k}(t)dx}^{1/2^k} \le \bra{\int_{\Omega}|u_0|^{2^k}dx + \frac{\widetilde{C}_k}{\delta_0}\bra{\sup_{s\in (0,t_0)}\int_{\Omega}u^{2^{k-1}}(s)dx}^2 + \frac{C_2}{\delta_0}}^{1/2^k}
\end{equation*}
By denoting $Q_k$ the left hand side of this inequality, we get
\begin{equation*} 
	Q_k \le \bra{\|u_0\|_{L^{\infty}}^{2^k}|\Omega| + \frac{\widetilde{C}_k}{\delta_0}Q_{k-1}^{2^k} + \frac{C_2}{\delta_0}}^{1/2^k}.
\end{equation*}
Thus, with $C_3 = |\Omega| + \delta_0^{-1} + C_2\delta_0^{-1} + 1$, we have
\begin{equation*} 
	\max\{Q_k, \|u_0\|_{L^{\infty}(\Omega)},1\} \le \max\{Q_{k-1}, \|u_0\|_{L^{\infty}(\Omega)},1\}C_3^{1/2^k}\widetilde{C}_k^{1/2^k}.
\end{equation*}
Thus, by $\widetilde{C}_k \sim (2^{n+2})^k$
\begin{equation*} 
	\begin{aligned}
		\max\{Q_k, \|u_0\|_{L^{\infty}(\Omega)},1\} &\le \max\{Q_0, \|u\|_{L^{\infty}(\Omega)},1\}\prod_{k=1}^{\infty}C_3^{1/2^k}\widetilde{C}_k^{1/2^k}\\
		&\le C\max\{\|u\|_{L^\infty(0,t_0;L^1(\Omega))}; \|u_0\|_{L^{\infty}(\Omega)}; 1\}C_3^{\sum_{k\ge 1}(1/2^k)}(2^{n+2})^{\sum_{k\ge 1}(k/2^k)}\\
		&\le C_4\max\{\|u\|_{L^\infty(0,t_0;L^1(\Omega))}; \|u_0\|_{L^{\infty}(\Omega)}; 1\}
	\end{aligned}
\end{equation*}
with $C_4$ depending only on $a, b, c$, $\|d\|_{L^{\infty}(S\times(0,T))}$, $\Omega$ and $n$. Letting $k\to \infty$, we obtain finally the desired estimate of the lemma.

\newcommand{\etalchar}[1]{$^{#1}$}


\begin{thebibliography}{KMTP94}
	
	\bibitem[Ada93]{AD1993}
	John~A Adam.
	\newblock The dynamics of growth-factor-modified immune response to cancer
	growth: One dimensional models.
	\newblock {\em Mathematical and Computer modelling}, 17(3):83--106, 1993.
	
	\bibitem[Ali79]{AL1979}
	Nicholas~D Alikakos.
	\newblock An application of the invariance principle to reaction-diffusion
	equations.
	\newblock {\em Journal of Differential Equations}, 33(2):201--225, 1979.
	
	\bibitem[ASR17]{ASR2017}
	Fatemeh Ansarizadeh, Manmohan Singh, and David Richards.
	\newblock Modelling of tumor cells regression in response to chemotherapeutic
	treatment.
	\newblock {\em Applied Mathematical Modelling}, 48:96--112, 2017.
	
	\bibitem[ATH{\etalchar{+}}24]{anderson2024global}
	Hannah~G. Anderson, Gregory~P. Takacs, Duane~C. Harris, Yang Kuang, Jeffrey~K.
	Harrison, and Tracy~L. Stepien.
	\newblock Global stability and parameter analysis reinforce therapeutic targets
	of {PD-L1-PD-1} and {MDSCs} for glioblastoma.
	\newblock {\em Journal of Mathematical Biology}, 88(10):33 pages, 2024.
	
	\bibitem[BP00]{BP2000}
	Nicola Bellomo and Luidgi Preziosi.
	\newblock Modelling and mathematical problems related to tumor evolution and
	its interaction with the immune system.
	\newblock {\em Mathematical and Computer Modelling}, 32(3-4):413--452, 2000.
	
	\bibitem[CSS92]{CSS1992}
	Joseph~J Casciari, Stratis~V Sotirchos, and Robert~M Sutherland.
	\newblock Mathematical modelling of microenvironment and growth in {EMT6/Ro}
	multicellular tumour spheroids.
	\newblock {\em Cell proliferation}, 25(1):1--22, 1992.
	
	\bibitem[DiB93]{DB1993}
	E~DiBenedetto.
	\newblock {\em Degenerate Parabolic Equations}.
	\newblock Springer-Verlag, 1993.
	
	\bibitem[DPR01]{DR2001}
	Lisette~G De~Pillis and Ami Radunskaya.
	\newblock A mathematical tumor model with immune resistance and drug therapy:
	an optimal control approach.
	\newblock {\em Computational and Mathematical Methods in Medicine},
	3(2):79--100, 2001.
	
	\bibitem[DPR03]{DR2003}
	Lisette~G De~Pillis and Ami Radunskaya.
	\newblock The dynamics of an optimally controlled tumor model: A case study.
	\newblock {\em Mathematical and computer modelling}, 37(11):1221--1244, 2003.
	
	\bibitem[ESJY21]{nikopoulou2021mathematical}
	Nikolopoulou Elpiniki, E.~Eikenberry Steffen, L.~Gevertz Jana, and Kuang Yang.
	\newblock Mathematical modeling of an immune checkpoint inhibitor and its
	synergy with an immunostimulant.
	\newblock {\em Discrete and Continuous Dynamical Systems - B},
	26(4):2133--2159, 2021.
	
	\bibitem[Eva10]{Evans}
	Lawrence~C Evans.
	\newblock {\em Partial Differential Equations}, volume~19.
	\newblock American Mathematical Soc., 2010.
	
	\bibitem[FGP99]{FGP1999}
	Bruno Firmani, Luciano Guerri, and Luigi Preziosi.
	\newblock Tumor/immune system competition with medically induced
	activation/deactivation.
	\newblock {\em Mathematical Models and Methods in Applied Sciences},
	9(04):491--512, 1999.
	
	\bibitem[FK11]{FK2011}
	Avner Friedman and Yangjin Kim.
	\newblock Tumor cells proliferation and migration under the influence of their
	microenvironment.
	\newblock {\em Mathematical Biosciences and Engineering}, 8(2):371--383, 2011.
	
	\bibitem[Fri05]{F2006}
	Avner Friedman.
	\newblock Cancer models and their mathematical analysis.
	\newblock In {\em Tutorials in Mathematical Biosciences III: Cell Cycle,
		Proliferation, and Cancer}, pages 223--246. Springer, 2005.
	
	\bibitem[Jac15]{WK2014}
	Trachette~L. Jackson.
	\newblock {\it {D}ynamics of cancer: mathematical foundations of oncology}
	[book review of {MR}3309233].
	\newblock {\em SIAM Rev.}, 57(1):161--162, 2015.
	
	\bibitem[KMTP94]{KMTP1994}
	Vladimir~A Kuznetsov, Iliya~A Makalkin, Mark~A Taylor, and Alan~S Perelson.
	\newblock Nonlinear dynamics of immunogenic tumors: parameter estimation and
	global bifurcation analysis.
	\newblock {\em Bulletin of mathematical biology}, 56(2):295--321, 1994.
	
	\bibitem[KP98]{KP1998}
	Denise Kirschner and John~Carl Panetta.
	\newblock Modeling immunotherapy of the tumor--immune interaction.
	\newblock {\em Journal of mathematical biology}, 37:235--252, 1998.
	
	\bibitem[LF17]{lai2017combination}
	Lai and A~Friedman.
	\newblock Combination therapy of cancer with cancer vaccine and immune
	checkpoint inhibitors: {A} mathematical model.
	\newblock {\em PLoS ONE}, 12(5):e0178479, 2017.
	
	\bibitem[LFJ{\etalchar{+}}09]{LF2010}
	John~S Lowengrub, Hermann~B Frieboes, Fang Jin, Yao-Li Chuang, Xiangrong Li,
	Paul Macklin, Steven~M Wise, and Vittorio Cristini.
	\newblock Nonlinear modelling of cancer: bridging the gap between cells and
	tumours.
	\newblock {\em Nonlinearity}, 23(1):R1, 2009.
	
	\bibitem[Lie96]{Lieberman}
	Gary~M Lieberman.
	\newblock {\em Second order parabolic differential equations}.
	\newblock World scientific, 1996.
	
	\bibitem[LN96]{NI}
	Yuan Lou and Wei-Ming Ni.
	\newblock Diffusion, self-diffusion and cross-diffusion.
	\newblock {\em Journal of Differential Equations}, 131(1):79--131, 1996.
	
	\bibitem[LSU68]{LSU}
	Olga~Aleksandrovna Ladyzhenskaia, Vsevolod~Alekseevich Solonnikov, and Nina~N
	Ural'tseva.
	\newblock {\em Linear and quasi-linear equations of parabolic type}, volume~23.
	\newblock American Mathematical Soc., 1968.
	
	\bibitem[Nit13]{RN}
	Robin Nittka.
	\newblock Quasilinear elliptic and parabolic robin problems on lipschitz
	domains.
	\newblock {\em Nonlinear Differential Equations and Applications NoDEA},
	20(3):1125--1155, 2013.
	
	\bibitem[OS99]{OS1999}
	Markus~R Owen and Jonathan~A Sherratt.
	\newblock Mathematical modelling of macrophage dynamics in tumours.
	\newblock {\em Mathematical Models and Methods in Applied Sciences},
	9(04):513--539, 1999.
	
	\bibitem[RCM07]{RCM2007}
	Tiina Roose, S~Jonathan Chapman, and Philip~K Maini.
	\newblock Mathematical models of avascular tumor growth.
	\newblock {\em SIAM review}, 49(2):179--208, 2007.
	
	\bibitem[Tem97]{Temam1988}
	Roger Temam.
	\newblock {\em Infinite-dimensional dynamical systems in mechanics and
		physics}, volume~68 of {\em Applied Mathematical Sciences}.
	\newblock Springer-Verlag, New York, second edition, 1997.
	
	\bibitem[Yin22]{Yin2022}
	Hong-Ming Yin.
	\newblock On a mathematical model arising from an optimal chemotherapeutic drug
	treatment for tumor cells.
	\newblock {\em arXiv:2212.05146}, 2022.
	
\end{thebibliography}
\end{document}